\documentclass[11pt,reqno]{amsart}
\usepackage{amssymb}
\usepackage{upgreek}
\usepackage{dsfont }
\usepackage{mathrsfs}
\usepackage{mathtools}
\usepackage[all]{xy}
\usepackage{color}
\usepackage{verbatim}
\allowdisplaybreaks
\newtheorem{thm}{Theorem}[section]
\newtheorem{lem}[thm]{Lemma}

\newtheorem{cor}[thm]{Corollary}

\theoremstyle{definition}
\newtheorem{dfn}[thm]{Definition}

\theoremstyle{remark}
\newtheorem{remark}[thm]{Remark}

\newcommand{\af}{\alpha}
\newcommand{\bt}{\beta}

\newcommand{\dt}{\delta}
\newcommand{\ep}{\varepsilon}

\newcommand{\ld}{\lambda}

\newcommand{\ph}{\varphi}

\newcommand{\R}{{\mathbb{R}}}

\newcommand{\Hess}{\mathrm{Hess}}
\newcommand{\nab}{\nabla}
\newcommand{\pr}{\partial}

\begin{document}


\title[Matrix inequality for the Laplace equation]
{Matrix inequality for the Laplace equation}

\author[Jiewon Park]{Jiewon Park}
\address{
Department of Mathematics\\
Massachusetts Institute of Technology\\
182 Memorial Drive, Cambridge, MA 02139\\
United States} \email{jiewon\-@\-mit.\-edu}



\begin{abstract}   
Since Li and Yau obtained the gradient estimate for the heat equation, related estimates have been extensively studied. With additional curvature assumptions, matrix estimates that generalize such estimates have been discovered for various time-dependent settings, including the heat equation on a K\"{a}hler manifold, Ricci flow, K\"{a}hler-Ricci flow, and mean curvature flow, to name a few. As an elliptic analogue, Colding proved a sharp gradient estimate for the Green function on a manifold with nonnegative Ricci curvature. In this paper we prove a related matrix inequality on manifolds with suitable curvature and volume growth assumptions. 
\end{abstract} 

\maketitle

\section{Introduction}

In the seminal paper \cite{LY}, Li and Yau proved a sharp estimate for the gradient of the heat kernel on a complete Riemannian manifold with Ricci curvature bounded below. It leads to a Harnack inequality on such manifolds by integration along shortest geodesics, which is sometimes referred to as a differential Harnack inequality. Later Hamilton \cite{H2} discovered a time-dependent matrix quantity that stays positive-semidefinite at all time, in the case that the manifold has nonnegative sectional curvature and parallel Ricci curvature. Taking the trace of this matrix inequality yields the Li-Yau gradient estimate. 

Matrix estimates have also been developed for other settings, such as the heat equation on K{\"a}hler manifolds with nonnegative holomorphic bisectional curvature by L. Ni and H. D. Cao \cite{CN}, Ricci flow by Hamilton \cite{H1}, K{\"a}hler-Ricci flow by L. Ni \cite{N1}, and mean curvature flow by Hamilton \cite{H3}, to name a few. There are close connections between such Harnack estimates and entropy formulae, as illustrated in the excellent survey by L. Ni \cite{N2}.

As an elliptic setting parallel to the aforementioned time-dependent results, Colding \cite{C} obtained a sharp gradient estimate for the minimal positive Green function for the Laplace equation under the relatively mild assumption of nonnegative Ricci curvature. This estimate is closely related to monotonicity formulae for manifolds with Ricci curvature bounded below; for details we refer to \cite{C}, \cite{CM}, and \cite{CM2}. Such monotone quantities turn out to be extremely useful, as they control the distance to the nearest cone of the manifold and are used  to prove the uniqueness of the tangent cone for Einstein manifolds \cite{CM3}. In this paper, we show that there exists a related matrix inequality.

\vskip 1pc

\begin{thm}\label{the thm}
Let $(M^n,g)$ be a complete non-compact Riemannian manifold of  Euclidean volume growth and dimension $n\geq 3$. Let $G$ be the minimal positive Green function with pole at $x\in M$ and $b=G^\frac{1}{2-n}$. Suppose that $M$ has nonnegative sectional curvature along $\nab G$ \footnote{By this we mean that $R(\nab G,V,\nab G, V)\geq 0$ for any vector field $V$ on $M$.} and parallel Ricci curvature. If $\Hess_{b^2}$ is uniformly bounded from above on $M\backslash\{x\}$ \footnote{By this we mean that $\Hess_{b^2} \leq Dg$ everywhere on $M\backslash\{x\}$ for some $D>0$.} and asymptotically bounded from above by $Cg$ at $x$,\footnote{By this we mean that 
\begin{align*}
\underset{\ep\rightarrow 0}{\liminf}\sup_{\substack{(p,V)\\r(p)=\ep,\hspace{0.6mm} V\in T_p M,\hspace{0.6mm} g(V,V)=1}}\big(\Hess_{b^2}-Cg\big)(V,V)\leq 0.
\end{align*} } where $C\geq 10$, then $\Hess_{b^2} \leq Cg$ holds everywhere on $M\backslash\{x\}$.
\end{thm}

\vskip 1pc

To motivate the above theorem, suppose for a moment that $M=\R^n$. The minimal positive Green function $G$ with a pole at the origin is given by $C(n)\cdot r^{2-n}$, where $r$ is the distance from the origin and $C(n)$ is a dimensional constant. We observe that the first and the second order derivatives of $G$ satisfies the following relation, which motivates a bound on the Hessian of $G$.
\begin{align*} 
G_{ij}+\frac{n}{2-n}\cdot\frac{G_iG_j}{G}=(2-n)G^\frac{-n}{2-n}\dt_{ij}.	
\end{align*}

Another motivation, which we describe here, comes from the Hessian comparison theorem for radial functions. Let $(M^n,g)$ be a complete non-compact Riemannian manifold. Fixing a point $x$, $M$ is called \textit{parabolic} if it does not admit a positive Green function for the Laplacian with pole at $x$. It is called \textit{non-parabolic} otherwise. If $M$ has nonnegative Ricci curvature, a result of Varopoulos \cite{V} states that $M$ is non-parabolic if and only if $\int_s^\infty \frac{t}{\mathrm{Vol}B_x(t)} dt<\infty$ for any positive $s$. $M$ is said to have \textit{Euclidean volume growth} if fore some $c>0$ it holds that $\mathrm{Vol}(B_x(t))\geq c\cdot t^n$ for any $t>0$. The result of Varopoulos implies that a manifold of dimension $\geq 3$ and Euclidean volume growth is non-parabolic. On such $M$, combining the results of Li and Tam \cite{LT} and Gilbarg and Serrin \cite{GS}, we see that there exists a unique minimal positive symmetric Green function $G=G(x,y)$ such that $G(x,y)=G_x(y)=O(r^{2-n})$, where $r$ is the distance from $x$. Hence, a bound on the Hessian of $G$ would be a natural analogue for the Hessian comparison theorem. From now on, we normalize $G$ suitably so that $G=r^{2-n}$ on the Euclidean space $M=\R^n$. Then we can define a function $b$ as follows.
\begin{equation*}
	b:=G^\frac{1}{2-n}.
\end{equation*} 

\noindent Then $b$ corresponds to just $r$, and therefore might be more intuitive than $G$. On the Euclidean space $\R^n$, the Hessian of $b^2=r^2$ satisfies the following equation.
\begin{equation*}
\Hess_{b^2}=2g.	
\end{equation*}

The Hessian comparison theorem would suggest a result in the direction of $\Hess_{b^2}\leq Cg$ with $2\leq C$, so we ask under which conditions on $M$ such a bound could be obtained. It turns out that if $M$ has nonnegative sectional curvature along $\nab G$ and parallel Ricci curvature, and if $\Hess_{b^2}$ is bounded locally near $x$ and also arbitrarily far out, then the bound extends globally to the region in between. The precise meaning of the last two conditions is the following. Suppose that $\Hess_{b^2} \leq Cg$ in the neighborhood of $x$ with the constant $C$ as in the theorem below. Then only one of the two cases can happen: either $\Hess_{b^2} \leq Cg$ on all of $M\backslash \{x\}$, or $\Hess_{b^2}$ diverges as $r \rightarrow \infty$. Such is the content of Theorem \ref{the thm}.

\vskip 1pc 
\begin{remark}
The two curvature assumptions in Theorem \ref{the thm} are critical in the proof, and were also taken in \cite{H2}. Note, however, that all of the arguments in the proof of Theorem \ref{the thm} can readily be generalized to the case where sectional curvature is bounded from below by $-K \cdot G^{-\frac{2}{2-n}}$ and the first derivative of Ricci curvature is bounded as $|\nab_i R_{jk}| \leq L\cdot G^{-\frac{3}{2-n}}$. Then we obtain an upper bound of $\Hess_{b^2}$ in terms of $n,K,L$. It would be an interesting question to ask whether a similar inequality holds under scale-invariant curvature assumptions with $r$ instead of $G$, i.e. under the assumptions that the sectional curvature is bounded from below by $-K \cdot r^{-2}$ and the first derivative of Ricci curvature is bounded as $|\nab_i R_{jk}| \leq L\cdot r^{-3}$.
\end{remark}

\vskip 1pc

As a corollary we obtain a Harnack inequality for $b$. Let $y,z\in M\backslash\{x\}$ and consider a minimal geodesic segment $\overline{yz}$ parametrized by arclength $s$. Then under the assumptions of Theorem \ref{the thm}, the function $\frac{C}{2}s^2-b^2$ is convex. Hence we obtain the following corollary.

\vskip 1pc

\begin{cor}
Under the assumptions of Theorem \ref{the thm}, let $w$ be the point on a minimal geodesic $\overline{yz}$ such that $d(y,w)=\ld\cdot d(y,z)$ and $d(w,z)=(1-\ld) \cdot d(y,z)$, $0\leq\ld\leq 1$. Then 
\begin{align*}
b(w)^2 \geq (1-\ld)\cdot b(y)^2+\ld\cdot b(z)^2-\frac{C}{2}\ld(1-\ld)\cdot d(y,z)^2.
\end{align*} 
\end{cor}

\vskip 1pc

Note that the corollary holds whether $b(y)\neq b(z)$ or not. This compares with the fact that integrating an estimate on the scalar quantity $|\nab b|$ can only compare values of $b$ between two points on different level sets of $b$. Also, note that in the above corollary, the Euclidean space $M=\R^n$ achieves the equality with $C=2$.

\vskip 1pc

\noindent \textbf{Acknowledgements.} The author would like to thank Prof. Tobias Colding for initially bringing the problem to the author's attention, and for numerous valuable discussions and suggestions.

\vskip 1pc

\section{Proof of the matrix inequality}

\noindent \textbf{Notation.} For the convention of the curvature tensor, we use 
\begin{equation*}
	R(X,Y,Z,W)=g(\nab_Y \nab_X Z-\nab_X \nab_Y Z+\nab_{[X,Y]}Z,W).
\end{equation*}

In an orthonormal frame $\{e_i\}$ we write in coordinates that $R(e_i,e_j,e_k,e_l)=R_{ijkl}$. The Ricci curvature is defined as $Ric(X,Y)=\sum_k R(X,e_k,Y,e_k)$ and denoted in coordinates as $R_{ij}=Ric(e_i,e_j)$. Repeated indices are understood as summations, unless otherwise specified. $\nab_{e_i}$ will often be abbreviated by $\nab_i$. 

\vskip 1pc 

In this section we present the proof of Theorem \ref{the thm}. The main tool is the maximum principle introduced by Calabi in \cite{Ca}, which we recall below.

\vskip 1pc

\begin{dfn}
Let $X$ be a Riemannian manifold, $x_0 \in X$, and $\ph: X \rightarrow \R$ a continuous function. We say that \textit{$\Delta\ph \leq 0$ at $x_0$ in barrier sense} if for any $\ep>0$, there is a $\mathcal{C}^2$ function $\psi_{x_0,\ep}$ on a neighborhood of $x_0$ such that $\psi_{x_0,\ep}(x_0) = \ph(x_0)$, $\Delta\psi_{x_0,\ep}<\ep$, and $\psi_{x_0,\ep}\geq\ph$. We say that \textit{$\Delta\ph \leq 0$ in barrier sense} if $\Delta\ph \leq 0$ at $x_0$ in barrier sense for all $x_0\in X$.
\end{dfn}

\vskip 1pc

\begin{lem}[Maximum principle for barrier subsolutions]
	If $\Delta\ph \leq 0$ in barrier sense, then either $\ph$ is constant or $\ph$ has no weak local minimum.
\end{lem}

\vskip 1pc

Define a tensor $H$ as the following, motivated by the fact that it vanishes on the Euclidean space.
\begin{equation*}
H=\Hess_G+\frac{n}{2-n}\cdot\frac{\nab G\otimes\nab G}{G}+(n-2)\cdot G^\frac{-n}{2-n}g.
\end{equation*}

By a straightforward computation, it follows that $\Hess_{b^2}=-\frac{2}{n-2}G^{\frac{n}{2-n}}H+2g$. Hence, the assumption that $\Hess_{b^2} \leq Dg$ is equivalent to $0\leq H+\frac{n-2}{2}(D-2)G^{\frac{-n}{2-n}}g$. Let $\af=-\frac{n}{2-n}=\frac{n}{n-2}$, so that our goal is to show that $0\leq H+\frac{n-2}{2}(C-2)G^\af g$. For convenience call this tensor $\widetilde{H}$,
\begin{align*}
	\widetilde{H}:=H+\frac{n-2}{2}(C-2)G^\af g=\Hess_G+\frac{n}{2-n}\cdot\frac{\nab G\otimes\nab G}{G}+\frac{n-2}{2}C\cdot G^\af g.
\end{align*}

We also define the function $\Lambda$ on $M\backslash \{x\}$ to be the lowest eigenvalue of $\widetilde{H}$,

\begin{align*}
	\Lambda(p) := \min_{V\in T_p M, \hspace{0.6mm}g(V,V)=1} \widetilde{H}(V,V).
\end{align*}

Then $\Lambda$ is continuous, and the assumption that $\Hess_{b^2}\leq Dg$ implies that $\frac{n-2}{2}(C-D)G^\af\leq \Lambda$. 

An ingredient we will need is the following lemma. Let $p\in M\backslash\{x\}$ and let $\{e_i\}$ be a normal frame at $p$, i.e. $g_{ij}(p)=\dt_{ij}$ and $\nab_j e_i(p)=0$ for any $i,j$. Denote $\widetilde{H}_{ij}=\widetilde{H}(e_i,e_j)$. For convenience, define the tensor $B$ as $B=\frac{\nab G\otimes \nab G}{G}$, or in coordinates as $B_{ij}=\frac{G_i G_j}{G}$. Then $B$ is positive-semidefinite with eigenvalues $\frac{|\nab G|^2}{G}$ and $0$. We compute the following quantity in a straightforward manner.

\vskip 1pc

\begin{lem}\label{lap of harnack}
The following holds at $p$.
\begin{align*}
\Delta&(\widetilde{H}_{ij}-\frac{n-2}{2}C\cdot G^\af g_{ij})\\
&=R_{ik}\widetilde{H}_{jk}+R_{jk}\widetilde{H}_{ik}-2R_{ikjl}\widetilde{H}_{kl}-\frac{2n}{n-2}R_{ikjl}\frac{G_iG_j}{G}-\frac{2n}{(n-2)G}\widetilde{H}^2_{ij}\\
&\hspace{5mm}-\frac{n(n-2)}{2}C^2G^{2\af-1}g_{ij}+\frac{4n}{n-2}\bigg[C\cdot G^{\af-1}-\frac{2|\nab G|^2}{(n-2)^2 G^2}\bigg]B_{ij}\\
&\hspace{5mm}+\frac{2n}{(n-2)G}\bigg[\widetilde{H}\bigg(\frac{2}{2-n}B+\frac{n-2}{2}C\cdot G^\af g\bigg)+\bigg(\frac{2}{2-n}B+\frac{n-2}{2}C\cdot G^\af g\bigg)\widetilde{H}\bigg]_{ij}.
\end{align*}	
\end{lem}

\vskip 1pc

The proof of this fact is provided in Section 3.

\vskip 1pc

\begin{proof}[Proof of Theorem \ref{the thm}]

First we note that if $\Delta\Lambda \leq 0$	in barrier sense whenever $\Lambda<0$, then the theorem would follow by the maximum principle. Indeed, in the case that $\Lambda$ is constant, note that $\Lambda \geq \frac{n-2}{2}(C-D)\cdot G^\af$ and $G^\af=O(r^{-n})$, therefore $\Lambda\geq 0$. In the case that $\Lambda$ is not constant, $\Lambda$ takes its negative minimum on $\{\ep\leq r\leq R\}$ on the boundary by the maximum principle. By the same argument as in the constant case, we have that $\underset{r=R}{\inf}\Lambda \rightarrow  0$ as $R \rightarrow \infty$, and the assumption near $x$ implies that $\liminf_{\ep \rightarrow 0} \inf_{r=\ep} \Lambda \geq 0$. Therefore it would suffice to establish that $\Delta\Lambda \leq 0$ whenever $\Lambda<0$.

Now suppose that $\Lambda(p)=\widetilde{H}(V,V) < 0$. Write $V=V^ie_i$ on a neighborhood of $p$, where each $V^i$ is extended as a constant function. Define $\widetilde{h}=\widetilde{H}(V,V)=\widetilde{H}_{ij}V^iV^j$. We observe that $\widetilde{h}$ is an upper barrier for $\Lambda$ at $p$. Indeed, $\widetilde{h}(p) = \Lambda(p)$ and $\widetilde{h} \geq \Lambda$ near $p$ by definition of $\Lambda$. It only remains to show that, for any $\ep>0$, if we choose the neighborhood of $p$ small enough then $\Delta\widetilde{h}< \ep$. It is enough to show that if $\widetilde{h}(p)<0$ then $\Delta(\widetilde{H}_{ij}V^iV^j)(p)\leq 0$, since then $\Delta\widetilde{h}<\ep$ follows by continuity. Hence in what follows, all computations are made at $p$. Note that since $V^i$ are constant, we have that $\Delta\widetilde{h}=\Delta(\widetilde{H}_{ij})V^iV^j=(\Delta \widetilde{H}_{ij})V^iV^j$. Thus, it suffices to estimate the terms in Lemma \ref{lap of harnack}.

We bound the first three terms related to the curvature in the following way. Without loss of generality we can assume that $\{e_i\}$ diagonalizes $\widetilde{H}$ at $p$ and write $\widetilde{H}_{ij}=\ld_i\dt_{ij}$. Since $V$ is the lowest eigenvector of $\widetilde{H}$, there is $m$ such that $V=e_m$ with $\lambda_m=\Lambda$. Therefore (with $m$ fixed and $i,j,k,l$ being summed over),
\begin{align*}
&\hspace{2mm}\big(R_{ik}\widetilde{H}_{jk} + R_{jk}\widetilde{H}_{ik}-2R_{ikjl}\widetilde{H}_{kl}\big)V^i V^j \\
&=R_{ik}(\widetilde{H}_{jk} V^j)V^i+R_{jk}(\widetilde{H}_{ik} V^i)V^j-2R_{ikjl}\lambda_k \dt_{kl} V^i V^j \\
&=R_{ik}(\Lambda\cdot V^k)V^i+R_{jk}(\Lambda\cdot V^k)V^j-2R_{ikjk}\lambda_k \dt_{im}\dt_{jm}\\
&=2\Lambda \cdot R_{ij}\dt_{im}\dt_{jm}-2R_{mkmk}\lambda_k\\
&=2R_{mkmk}(\Lambda-\lambda_k) \leq 0,
\end{align*}
since $\Lambda$ is the lowest eigenvalue, and $R_{mkmk}\geq 0$.

The assumption on the sectional curvature implies that $-R_{ikjl}\frac{G_k G_l}{G}V^iV^j \leq 0$. It is also clear that $-\frac{2n}{(n-2)G}(\widetilde{H})^2_{ij}V^iV^j \leq 0$.
 
For the next two of the remaining terms, we will use the sharp gradient estimate in \cite{C} which states that $|\nab b|\leq 1$ for nonnegative Ricci curvature. This is equivalent to $|\nab G|^2\leq (n-2)^2 G^{\af+1}$. Therefore,
\begin{align*}
&-\frac{n(n-2)}{2}C^2G^{2\af-1}g_{ij}V^iV^j+\frac{4n}{n-2}\bigg[C\cdot G^{\af-1}-\frac{2|\nab G|^2}{(n-2)^2 G^2}\bigg]B_{ij}V^iV^j\\
&\leq -\frac{n(n-2)}{2}C^2G^{2\af-1}+\frac{4nC\cdot G^{\af-1}}{n-2}B_{ij}V^iV^j \\
&\leq -\frac{n(n-2)}{2}C^2G^{2\af-1}+\frac{4nC\cdot G^{\af-2}}{n-2}|\nab G|^2\\
&\leq -\frac{n(n-2)}{2}C^2G^{2\af-1}+\frac{4nC\cdot G^{\af-2}}{n-2}\cdot (n-2)^2G^{\af+1}\\
&= -\frac{n(n-2)}{2}C(C-8)G^{2\af-1}.
\end{align*}

For the last group of terms, we use that the top eigenvalue of $B$ is $\frac{|\nab G|^2}{G}$ and the gradient estimate $|\nab G|^2\leq (n-2)^2G^{\af+1}$ to obtain that
\begin{align*}
\bigg[&\widetilde{H}\bigg(\frac{2}{2-n}B+\frac{n-2}{2}C\cdot G^\af g\bigg)+\bigg(\frac{2}{2-n}B+\frac{n-2}{2}C\cdot G^\af g\bigg)\widetilde{H}\bigg]_{ij}V^i V^j\\
&=\frac{2}{2-n}[\widetilde{H}B+B\widetilde{H}]_{ij}V^iV^j+(n-2)C\cdot G^\af\widetilde{H}_{ij}V^iV^j\\
&\leq \frac{4|\nab G|^2}{(n-2)G}|\widetilde{h}|+(n-2)C\cdot G^\af\widetilde{h}\\
&=\bigg[(n-2)C\cdot G^\af-\frac{4|\nab G|^2}{(n-2)G}\bigg]\widetilde{h}\\
&=(n-2)(C-4)\cdot G^\af\widetilde{h}+\frac{4}{(n-2)G}\bigg[(n-2)^2G^{\af+1}-|\nab G|^2\bigg]\widetilde{h} \leq 0.
\end{align*}

Combining all of the above, we conclude that 
\begin{align*}
\big(\Delta(\widetilde{H}_{ij}-\frac{n-2}{2}C\cdot G^\af g_{ij})\big)V^iV^j\leq -\frac{n(n-2)}{2}C(C-8)G^{2\af-1}.
\end{align*}

$G^\af$ can be shown to satisfy the equation $\Delta(G^\af)=\frac{2n}{(n-2)^2}G^{\af-2}|\nab G|^2$. A proof of this fact is given in Section 3, Lemma \ref{misc}. Since $C\geq 10$, it follows that
\begin{align*}
\Delta (\widetilde{H}_{ij}V^iV^j)&=\Delta\bigg((\widetilde{H}_{ij}-\frac{n-2}{2}C\cdot G^\af g_{ij}+\frac{n-2}{2}C\cdot G^\af g_{ij})V^iV^j\bigg)\\
&=\Delta\bigg((\widetilde{H}_{ij}-\frac{n-2}{2}C\cdot G^\af g_{ij})V^iV^j\bigg)+\frac{(n-2)C}{2}\cdot\Delta G^\af\\
&=\Delta\bigg((\widetilde{H}_{ij}-\frac{n-2}{2}C\cdot G^\af g_{ij})V^iV^j\bigg)+\frac{nC}{n-2}\cdot G^{\af-2}|\nab G|^2\\
&\leq -\frac{n(n-2)}{2}C(C-8)G^{2\af-1}+\frac{nC}{n-2}\cdot G^{\af-2}|\nab G|^2\\
&\leq -\frac{n(n-2)}{2}C(C-8)G^{2\af-1}+n(n-2)C\cdot G^{2\af-1}\\
&=-\frac{n(n-2)}{2}C(C-10)G^{2\af-1}\\
&\leq 0,
\end{align*}
where the gradient estimate for $G$ was used for the second inequality. This establishes that $\Delta (\widetilde{H}_{ij}V^i V^j)\leq 0$ and finishes the proof of Theorem \ref{the thm}.
\end{proof}

\vskip 1pc

\begin{remark}
In \cite{H2} it is shown that for a positive solution $f$ of the heat equation on a closed manifold, the matrix quantity $\Hess_f-\frac{\nab f\otimes\nab f}{f}+\frac{f}{2t}g$ is positive-semidefinite for all time. One could ask whether we can introduce a cutoff function to view $G$ as a stationary solution on an annulus in $M$, and obtain the same result for $\Hess_G-\frac{\nab G\otimes\nab G}{G}+\frac{G}{2t}g$, which would imply that $\Hess_{b^2} \leq 4g$. However this setting seems ill-adapted to Hamilton's matrix maximum principle argument, as the assumption that $\pr M=\emptyset$ is essential there.
\end{remark}

\vskip 1pc

\section{Laplacian of the Harnack quantity}

This section is devoted to deriving Lemma \ref{lap of harnack}. We recall the commutators in the case of parallel Ricci curvature. 

\vskip 1pc

\begin{lem}\label{commutators}
Let $\{e_i\}$ be a normal frame at $p$. If $M$ has parallel Ricci curvature, i.e. $\nab_{i} R_{jk}=0$, then for a smooth function $f$ on $M$, the following identities hold at $p$.
\begin{align*}
f_{ij}&=f_{ji}, \\
f_{ijk}-f_{ikj}&=R_{jkli}f_l, \\
\Delta f_i - (\Delta f)_i &= R_{ik}f_k,\\
f_{ijkl}-f_{ijlk}&=R_{klmj}f_{im}+R_{klmi}f_{jm}, \\
\Delta f_{ij} - (\Delta f)_{ij}&=R_{jk}f_{ik}+R_{ik}f_{jk}-2R_{ikjl}f_{kl},
\end{align*}	
where $f_{i_1 i_2 \cdots i_k}$ just means the derivative $e_{i_k} (\cdots e_{i_2} (e_{i_1} (f))\cdots )$.
\end{lem}

\vskip 1pc

\begin{proof}
The first identity is the symmetry of the Hessian of $f$. For the second one, we compute that 
\begin{align*}
f_{ijk}-f_{ikj}&=e_k(g(\nab_j \nab f,e_i))-e_j(g(\nab_k \nab f,e_i))\\
&=g(\nab_k \nab_j\nab f,e_i)+g(\nab_j \nab f, \nab_k e_i)-g(\nab_j \nab_k\nab f,e_i)-g(\nab_k \nab f, \nab_j e_i)\\
&=g(\nab_k \nab_j\nab f,e_i)-g(\nab_j \nab_k\nab f,e_i) \\
&=R(e_j,e_k,\nab f,e_i)=R_{jkli}f_l.
\end{align*}

A similar identity holds for any 1-form $S$ in place of $df$, namely,
\begin{equation} 
(\nab^2 S)(e_i,e_j,e_k)-(\nab^2 S)(e_j,e_i,e_k)=S(R(e_j,e_i)e_k).\label{one form}	
\end{equation}

This can be checked in the same manner. We will use (\ref{one form}) to prove the fourth identity.

The third identity is a contraction of the one above,
\begin{align*}
f_{ikk}-f_{kki}=f_{kik}-f_{kki}=R_{iklk}f_l=R_{il}f_l=R_{ik}f_k.	
\end{align*}

 The fourth identity is actually true for any (0,2)-tensor $T$ in the following form.
\begin{equation*}
(\nab^2 T)(e_l,e_k,e_i,e_j)-(\nab^2 T)(e_k,e_l,e_i,e_j)=R_{klmj}T(e_i,e_m)+R_{klmi}T(e_m,e_j).	
\end{equation*}

\noindent To show this, let $T=T_1\otimes T_2$ for 1-forms $T_1$ and $T_2$, and compute using (\ref{one form}) and the normality of the coordinates, that
\begin{align*}
(&\nab^2 T)(e_l,e_k,e_i,e_j)-(\nab^2 T)(e_k,e_l,e_i,e_j)	\\
&=\nab^2 (T_1\otimes T_2)(e_l,e_k,e_i,e_j)-\nab^2 (T_1\otimes T_2)(e_k,e_l,e_i,e_j)\\
&=e_l(e_k(T_1(e_i)T_2(e_j)))-e_l(T_1(\nab_k e_i)T_2(e_j)+T_1(e_i)T_2(\nab_k e_j))\\
&\hspace{5mm}-e_k(e_l(T_1(e_i)T_2(e_j)))+e_k(T_1(\nab_l e_i)T_2(e_j)+T_1(e_i)T_2(\nab_l e_j))\\
&=-[e_l(T_1(\nab_k e_i))-e_k(T_1(\nab_l e_i))]T_2(e_j)-T_1(e_i)[e_l(T_2(\nab_k e_j))-e_k(T_2(\nab_l e_j))]\\
&=-[\nab^2 T_1(e_l,e_k,e_i)-\nab^2 T_1(e_k,e_l,e_i)]T_2(e_j)-T_1(e_i)[\nab^2 T_2(e_l,e_k,e_j)-\nab^2 T_2(e_k,e_l,e_j)]\\
&=-T_1(R(e_k,e_l)e_i)T_2(e_j)-T_1(e_i)T_2(R(e_k,e_l)e_j)\\
&=-R_{klim}T(e_m,e_j)-R_{kljm}T(e_i,e_m)\\
&=R_{klmj}T(e_i,e_m)+R_{klmi} T(e_m,e_j).
\end{align*}

\noindent Now the fourth identity follows from taking $T=\Hess_f$, and using the symmetry of the Hessian and the normality of the coordinates.

For the last identity, note that
\begin{align*}
\Delta f_{ij}&=f_{ijkk}=(f_{ikj}+R_{jkli}f_l)_k\\
&=f_{ikjk}+(\nab_k R_{jkli})f_l + R_{jkli}f_{kl}\\
&=f_{ikkj}+R_{jkmi}f_{km}+R_{jkmk}f_{mi}+(\nab_k R_{jkli})f_l + R_{jkli}f_{kl}\\
&=f_{kikj}-R_{ikjm}f_{km}+R_{jm}f_{im}+(\nab_k R_{jkli})f_l-R_{ikjl}f_{kl} \\
&=(f_{kki}+R_{il}f_l)_j-2R_{ikjl}f_{kl}+R_{jk}f_{ik}+(\nab_k R_{jkli})f_l \\
&=(\Delta f)_{ij}+R_{il}f_{jl}+R_{jk}f_{ik}-2R_{ikjl}f_{kl}+(\nab_k R_{jkli})f_l.
\end{align*}

\noindent The second Bianchi identity implies that
\begin{align*}
\nab_k R_{jkli}+\nab_l R_{jkik}+\nab_i R_{jkkl} = \nab_k R_{jkli}+\nab_l R_{ji} -\nab_i R_{jl}=0.	
\end{align*}

\noindent Since $M$ has parallel Ricci curvature, it follows that $\nab_k R_{jkli}=0$. Thus we arrive at
\begin{align*}
\Delta f_{ij}&=(\Delta f)_{ij}+R_{il}f_{jl}+R_{jk}f_{ik}-2R_{ikjl}f_{kl}.
\end{align*}

\noindent Changing $k$ and $l$ suitably, we have shown the lemma.
\end{proof}

\vskip 1pc

With Lemma \ref{commutators} we compute the ingredients for $\Delta\widetilde{H}_{ij}$, additionally using only the Leibniz rule.

\vskip 1pc

\begin{lem}\label{misc} Let $\{e_i\}$ be a normal frame at $p$, and suppose that $M$ has parallel Ricci curvature. Then the following identities hold at $p$.
\begin{align*}
\Delta G_{ij}&=	R_{jk}G_{ik}+R_{ik}G_{jk}-2R_{ikjl}G_{kl}, \\
\Delta(G_i G_j)&=R_{ik}G_jG_k+R_{jk}G_iG_k+2G_{ik}G_{jk},\\
g(\nab G,\nab(G_iG_j))&=G_iG_kG_{jk}+G_jG_kG_{ik}, \\
\Delta\bigg(\frac{G_iG_j}{G}\bigg)&=R_{ik}\frac{G_jG_k}{G}+R_{jk}\frac{G_iG_k}{G}\\
&\hspace{5mm}+\frac{2G_{ik}G_{jk}}{G}+\frac{2|\nab G|^2G_iG_j}{G^3}-\frac{2G_k(G_iG_{jk}+G_jG_{ik})}{G^2},\\
\Delta G^\af &=\frac{2n}{(2-n)^2}G^{\af-2}|\nab G|^2.
\end{align*}
\end{lem}

\vskip 1pc

\begin{proof}
The first identity is immediate from Lemma \ref{commutators} and the fact that $\Delta G=0$, and the third identity is an application of the Leibniz rule on $G_i G_j$. For the second identity,
\begin{align*}
\Delta(G_iG_j)&= \Delta(G_i)G_j+G_j\Delta(G_i)+2G_{ik}G_{jk} \\
&=[(\Delta G)_i+R_{ik}G_k]G_j+[(\Delta G)_j+R_{jk}G_k]G_i+2G_{ik}G_{jk} \\
&=R_{ik}G_jG_k+R_{jk}G_iG_k+2G_{ik}G_{jk}.
\end{align*}

 We also derive that for any $\bt$,
\begin{align*}
\Delta G^{\bt}&=\text{div}(\bt \cdot G^{\bt-1}\nab G)=\bt(\bt-1)G^{\bt-2}|\nab G|^2,
\end{align*}

\noindent from which the last identity is immediate and it follows that $\Delta(G^{-1})=2G^{-3}|\nab G|^2$. We use this and the third identity to check the fourth identity,
\begin{align*}
\Delta\bigg(\frac{G_iG_j}{G}\bigg)&=	\frac{\Delta(G_iG_j)}{G}+\Delta(G^{-1})G_iG_j-\frac{2}{G^2}g\big(\nab G,\nab(G_iG_j)\big)\\
&=\frac{\Delta(G_iG_j)}{G}+\frac{2|\nab G|^2G_iG_j}{G^3}-\frac{2G_k(G_iG_{jk}+G_jG_{ik})}{G^2}.
\end{align*}
\end{proof}

\vskip 1pc

\begin{lem} \label{B squared}
Let $B=\frac{\nab G\otimes \nab G}{G}$, or equivalently in coordinates, $B_{ij}=\frac{G_i G_j}{G}$ for an orthonormal frame $\{e_i\}$. Then 
$B^2=\frac{|\nab G|^2}{G}B$.
\end{lem}

\begin{proof}
\begin{align*}
	(B^2)_{ij}=\frac{G_iG_k\cdot G_jG_k}{G^2}=\frac{|\nab G|^2}{G}\cdot \frac{G_iG_j}{G}=\frac{|\nab G|^2}{G}B_{ij}.
\end{align*}	
\end{proof}

We are now ready to prove Lemma \ref{lap of harnack}.

\vskip 1pc 

\begin{proof}[Proof of Lemma \ref{lap of harnack}]
By Lemma \ref{misc}, we have that
\begin{align*}
\Delta&(\widetilde{H}_{ij}-\frac{n-2}{2}C\cdot G^\af g_{ij})\\
&=\Delta\bigg(G_{ij}+\frac{n}{2-n}\cdot \frac{G_iG_j}{G}\bigg)\\
&=R_{jk}G_{ik}+R_{ik}G_{jk}-2R_{ikjl}G_{kl}+\frac{n}{2-n}\bigg(\frac{R_{ik}G_jG_k+R_{jk}G_iG_k}{G}\\
&\hspace{38mm}+\frac{2G_{ik}G_{jk}}{G}+\frac{2|\nab G|^2G_iG_j}{G^3}-\frac{2G_k(G_iG_{jk}+G_jG_{ik})}{G^2}\bigg)\\
&=R_{ik}\bigg(G_{jk}+\frac{n}{2-n}\cdot\frac{G_jG_k}{G}\bigg)+R_{jk}\bigg(G_{ik}+\frac{n}{2-n}\cdot\frac{G_iG_k}{G}\bigg)-2R_{ikjl}G_{kl}\\
&\hspace{38mm}+\frac{2n}{(2-n)G}\big[\Hess_G-B\big]^2_{ij}.
\end{align*}

\noindent Substituting the derivatives of $G$ with expressions in $\widetilde{H}$, we obtain that
\begin{align*}
\Delta&(\widetilde{H}_{ij}-\frac{n-2}{2}C\cdot G^\af g_{ij})\\
&=R_{ik}\bigg(\widetilde{H}_{jk}-\frac{n-2}{2}C\cdot G^\af g_{jk}\bigg)+R_{jk}\bigg(\widetilde{H}_{ik}-\frac{n-2}{2}C\cdot G^\af g_{ik}\bigg)\\
&\hspace{5mm}-2R_{ikjl}\bigg(\widetilde{H}_{kl}-\frac{n}{2-n}\frac{G_iG_j}{G}-\frac{n-2}{2}C\cdot G^\af g_{kl}\bigg)\\
&\hspace{5mm}+\frac{2n}{(2-n)G}\bigg[\widetilde{H}-\frac{n}{2-n}B-\frac{n-2}{2}C\cdot G^\af g - B\bigg]^2_{ij}.\\
\end{align*}

\noindent We expand the square term and rearrange as follows.
\begin{align*}
\Delta&(\widetilde{H}_{ij}-\frac{n-2}{2}C\cdot G^\af g_{ij})\\
&=R_{ik}\widetilde{H}_{jk}+R_{jk}\widetilde{H}_{ik}-2R_{ikjl}\widetilde{H}_{kl}-\frac{2n}{n-2}R_{ikjl}\frac{G_iG_j}{G}\\
&\hspace{5mm}-\frac{2n}{(n-2)G}\bigg[\widetilde{H}-\frac{2}{2-n}B-\frac{n-2}{2}C\cdot G^\af g\bigg]^2_{ij}\\
&=R_{ik}\widetilde{H}_{jk}+R_{jk}\widetilde{H}_{ik}-2R_{ikjl}\widetilde{H}_{kl}-\frac{2n}{n-2}R_{ikjl}\frac{G_iG_j}{G}-\frac{2n}{(n-2)G}(\widetilde{H})^2_{ij}\\
&\hspace{5mm}-\frac{8n}{(n-2)^3G}(B^2)_{ij}-\frac{n(n-2)}{2}C^2G^{2\af-1}g_{ij}+\frac{4n}{n-2}C\cdot G^{\af-1} B_{ij}\\
&\hspace{5mm}+\frac{2n}{(n-2)G}\bigg[\widetilde{H}\bigg(\frac{2}{2-n}B+\frac{n-2}{2}C\cdot G^\af g\bigg)+\bigg(\frac{2}{2-n}B+\frac{n-2}{2}C\cdot G^\af g\bigg)\widetilde{H}\bigg]_{ij}.
\end{align*}

\noindent Replacing $B^2$ with $\frac{|\nab G|^2}{G}B$ by Lemma \ref{B squared} and rearraging, it follows that
\begin{align*}
\Delta&(\widetilde{H}_{ij}-\frac{n-2}{2}C\cdot G^\af g_{ij})\\
&=R_{ik}\widetilde{H}_{jk}+R_{jk}\widetilde{H}_{ik}-2R_{ikjl}\widetilde{H}_{kl}-\frac{2n}{n-2}R_{ikjl}\frac{G_iG_j}{G}-\frac{2n}{(n-2)G}(\widetilde{H})^2_{ij}\\
&\hspace{5mm}-\frac{8n}{(n-2)^3}\frac{|\nab G|^2}{G^2}B_{ij}-\frac{n(n-2)}{2}C^2G^{2\af-1}g_{ij}+\frac{4n}{n-2}C\cdot G^{\af-1} B_{ij}\\
&\hspace{5mm}+\frac{2n}{(n-2)G}\bigg[\widetilde{H}\bigg(\frac{2}{2-n}B+\frac{n-2}{2}C\cdot G^\af g\bigg)+\bigg(\frac{2}{2-n}B+\frac{n-2}{2}C\cdot G^\af g\bigg)\widetilde{H}\bigg]_{ij}\\
&=R_{ik}\widetilde{H}_{jk}+R_{jk}\widetilde{H}_{ik}-2R_{ikjl}\widetilde{H}_{kl}-\frac{2n}{n-2}R_{ikjl}\frac{G_iG_j}{G}-\frac{2n}{(n-2)G}(\widetilde{H})^2_{ij}\\
&\hspace{5mm}-\frac{n(n-2)}{2}C^2G^{2\af-1}g_{ij}+\frac{4n}{n-2}\bigg[C\cdot G^{\af-1}-\frac{2|\nab G|^2}{(n-2)^2 G^2}\bigg]B_{ij}\\
&\hspace{5mm}+\frac{2n}{(n-2)G}\bigg[\widetilde{H}\bigg(\frac{2}{2-n}B+\frac{n-2}{2}C\cdot G^\af g\bigg)+\bigg(\frac{2}{2-n}B+\frac{n-2}{2}C\cdot G^\af g\bigg)\widetilde{H}\bigg]_{ij}.
\end{align*}

\noindent This finishes the proof of Lemma \ref{lap of harnack}.
\end{proof}

\vskip 2pc


\end{document}